\documentclass[11pt, oneside]{amsart}

\usepackage{geometry}
\usepackage[titletoc,toc,title]{appendix}
\usepackage{amsfonts}
\usepackage{amsmath}
\usepackage{amsthm}
\usepackage{euscript}
\usepackage{amssymb}
\usepackage{enumerate}
\usepackage{hyperref}
\usepackage{mathrsfs}
\usepackage{xypic}
\usepackage{stmaryrd}
\usepackage{mathdots}
\xyoption{all}

\theoremstyle{plain}
\newtheorem{thm}{Theorem}[section]
\newtheorem{cor}[thm]{Corollary}
\newtheorem{lem}[thm]{Lemma}

\newtheorem{prop}[thm]{Proposition}
\theoremstyle{definition}
\newtheorem{defn}{Definition}[section]
\theoremstyle{remark}
\newtheorem{rem}{Remark}[section]
\newtheorem{ex}{Example}[section]
\numberwithin{equation}{section}

\newcommand{\CC}{\mathbb C}
\newcommand{\N}{\mathbb N}

\newcommand{\GG}{\mathbb G}
\newcommand{\PP}{{\mathbb P}}

\newcommand{\OO}{\mathcal{O}}

\newcommand{\II}{\mathcal{I}}

\newcommand{\EE}{\mathcal{E}}

\newcommand{\lf}{\lfloor}
\newcommand{\rf}{\rfloor}

\newcommand{\ra}{\rightarrow}
\newcommand{\epi}{\twoheadrightarrow}

\newcommand{\op}{\oplus}
\newcommand{\OP}{\OO_{\PP^2}}
\newcommand{\TP}{\T_{\PP^2}}

\DeclareMathOperator{\HH}{H} 
\DeclareMathOperator{\hh}{h}
\DeclareMathOperator{\Ext}{Ext}

\DeclareMathOperator{\rk}{rk}
 
\DeclareMathOperator{\coker}{Coker} 

\DeclareMathOperator{\T}{T}

\begin{document}

\title[Planes of matrices of constant rank and gg vector bundles]{Planes of matrices of constant rank and globally generated vector bundles}

\author{Ada Boralevi}
\address{Scuola Internazionale Superiore di Studi Avanzati, via Bonomea 265, 34136 Trieste, Italy}
\curraddr{}
\email{ada.boralevi@sissa.it}

\author{Emilia Mezzetti}
\address{Dipartimento di Matematica e Geoscienze, Sezione di Matematica e Informatica,  Universit\`a di Trieste, Via Valerio 12/1, 34127 Trieste, Italy}
\curraddr{}
\email{mezzette@units.it}

\thanks{Research partially supported by MIUR funds, PRIN 2010-2011 project \lq\lq Geometria delle variet\`a algebriche'', 
and by Universit\`a degli Studi di Trieste -- FRA 2013 project \lq\lq Geometria e topologia delle variet\`a''.}



\begin{abstract}
We consider the problem of determining all pairs $(c_1, c_2)$ of Chern classes of rank $2$ bundles that are  cokernel 
of a skew-symmetric matrix of linear forms in $3$ variables, having constant rank $2c_1$ and size $2c_1+2$. We completely solve
the problem in the \lq\lq stable'' range, i.e. for pairs with $c_1^2-4c_2<0$, proving that the additional condition $c_2\leq {{c_1+1}\choose 2}$ is necessary and sufficient. 
For $c_1^2-4c_2\ge 0$, we prove
that there exist globally generated bundles, some even defining an
embedding of $\mathbb P^2$ in a Grassmannian, that cannot correspond to a matrix
of the above type. This extends previous work on $c_1\leq 3$.
\end{abstract}

\maketitle

\section{Introduction}
Even if it arises naturally in linear algebra, the problem of classifying linear systems of matrices of constant rank has many interactions with algebraic geometry. On the 
one hand, the understanding of these linear systems has greatly benefited from the use of algebraic geometry tools, as happened for example in 
\cite{sylvester,Westwick,Westwick1,bo_fa_me}. On the other side, linear systems have been proved useful in approaching some classical problems in geometry: examples of this phenomenon are, among many others, \cite{pirola,Ilic_JM,dp_m}. The main connection between the two areas comes from interpreting such a linear system as a vector bundles map, whose kernel and cokernel are again vector bundles on some projective variety. 

\vskip.05in

Let us explain more precisely the setting in which we work. Let $V$ be a vector space of dimension $n$ over $\CC$, and let 
$A \subseteq V \otimes V$ be a linear subspace of dimension $d$. Fixing bases, we can write down $A$ as a $n \times n$ 
matrix of linear forms in $d$ variables, that we denote by the same letter $A$.
We say that $A$ has \emph{constant rank} $r$ if every non-zero matrix obtained specializing the $d$ variables has rank $r$. 

The matrix $A$ can be viewed naturally as a map $V^* \otimes \OO_{\PP A}(-1)\xrightarrow{A} V \otimes \OO_{\PP A}$, and as such it gives an exact sequence:

\begin{equation}\label{estensione1}
  \xymatrix{0 \ar[r] & K \ar[r] & V^* \otimes \OO_{\PP A}(-1)\ar[r]^-A & V \otimes \OO_{\PP A} \ar[r]&E \ar[r]&0,}
\end{equation}
 
\noindent where both the kernel $K$ and the cokernel $E$ are vector bundles of rank $n-r$ on $\PP A$. 

A computation of invariants shows that there is a bound on the maximal dimension that the subspace $A$ can attain, 
namely for values of $2 \leq r \leq n$, such maximal dimension is comprised between $n-r+1$ and $2(n-r)+1$ \cite{Westwick1}. 
We stress the fact that these bounds are not effective in general. Moreover for a given value of $d$, only some values of $r$ are allowed. 

The further assumption that the subspace $A$ lies either in $S^2 V$ or in $\wedge^2 V$ 
yields a symmetry of the exact sequence \eqref{estensione1}, and gives an isomorphism $K=E^*(-1)$.  If that is the case, 
a similar computation of invariants as the one above shows us that the sequence \eqref{estensione1} 
determines the first Chern class $c_1(E)=\frac{r}{2}$. In particular, the rank $r=2c$ is even. 

We now want to focus on the special case $n=2c+2$, where the cokernel $E$ is a vector bundle of rank $2$ with $c_1(E)=c$, and where the maximal dimension 
of $A$ can vary between $3$ and $5$. From the fact  that $E$ has rank $2$ we deduce  that $E^*\simeq E(-c)$. 
All in all, if $A$ is either symmetric or skew-symmetric, then the exact sequence \eqref{estensione1} induced by $A$ can be rewritten as the following 2-step extension:
\begin{equation}\label{estensione2}
  \xymatrix{0 \ar[r] & E(-c-1) \ar[r] & \OO_{\PP A}(-1)^{2c+2}\ar[r]^-A & \OO_{\PP A}^{2c+2} \ar[r]&E \ar[r]&0.}
\end{equation}

For the symmetric case we refer the reader to \cite{Ilic_JM}, where it is shown that the effective bound for the dimension of 
linear spaces of symmetric matrices of co-rank $2$ is independent of $c$, and always equal to $3$. 

On the contrary, in the skew-symmetric case the bound does depend on the value $c$. If $c$ is odd, or if $c=2$, then the bound is equal to $3$, 
but for higher (even) values of $c$ it is in general not known. Almost all (that is, all except one) known examples of dimension $4$ 
were produced by the two authors together with D. Faenzi in \cite{bo_fa_me}, and they have $c=4$ and $6$. For these values it is easy to see that $4$ is an 
effective bound. It is conjectured that $4$ is in fact always an effective bound, and that $5$-dimensional examples do not exist.

\vskip.05in

In this paper we work on the skew-symmetric case with $\dim A=3$, so we deal with bundles on the projective plane. Recall that in this setting 
the sequence \eqref{estensione1} entails that $c_1(E)=c$, but it does not determine the value of $c_2(E)$. Hence it is quite natural to ask the following:

\vskip.05in

\textbf{Question.} \emph{Let $c$ be a positive integer. What are all the possible pairs $(c,y)$ such that there exists a skew-symmetric matrix of linear forms in $3$ variables, 
having constant rank $2c$, size $2c+2$, and cokernel $E$, with Chern classes $(c_1(E),c_2(E))=(c,y)$?}  

\vskip.05in

It is immediate from \eqref{estensione2} that the bundle $E$ is globally generated. Pairs $(c,y)$ such that there exists a globally generated rank $2$ vector bundle $F$ on $\PP^2$ 
with $c_1(F)=c$ and $c_2(F)=y$ are called \emph{effective}, and have been recently completely classified in \cite{ellia_effective}.

In a similar fashion, we call \emph{$m$-effective} a pair $(c,y)$ such that there exists a skew-symmetric matrix of linear forms in $3$ variables, 
having constant rank $2c$, size $2c+2$, and cokernel $E$, with Chern classes $(c_1(E),c_2(E))=(c,y)$. If this is the case, the vector bundle $E$ is also called $m$-effective.

Clearly $m$-effectiveness implies effectiveness, so our question can be rephrased by asking which effective pairs are also $m$-effective.

Remark that for low values of $c$ every globally generated bundle defining an embedding of $\PP^2$ in the Grassmannian is $m$-effective. More 
precisely, the case $c=1$ corresponds to the linear spaces contained in the Grassmannian of lines in $\PP^3$, hence it is classical. 
The cases $c=2$ and $c=3$ have been treated in \cite{Manivel_Mezzetti} and in \cite{Fania_Mezzetti} respectively. 
In particular in \cite{Manivel_Mezzetti} there is a complete classification of the orbits of linear spaces of $6\times 6$ skew-symmetric matrices of constant rank $4$, up to the natural action of the group $SL_6$. 

Here we tackle our question for general values of $c$. We begin in Section \ref{bounds} by proving the following upper bound: if the pair $(c,y)$ is $m$-effective then $0\leq y\leq {c+1 \choose 2}.$ 
Our main result (Theorem \ref{teo range stabile}) is a positive answer to our question for all pairs $(c,y)$ satisfying the previous condition and with $c^2-4y<0$, 
i.e. in the so-called \emph{stable range}. Moreover our answer is ``constructive'', in that for every pair $(c,y)$ we explicitly provide a vector bundle $E$ and its 
associated matrix $A$. 
Our method is somewhat similar to the construction used by Le Potier in \cite{le_potier_stabili}: 
we consider rank $2$ bundles that are \emph{quotients} (in a sense made precise in Definition \ref{def quoziente}) of direct sums of bundles 
of the form $\OP(i)$, $i\geq 1$, and $Q$, where $Q$ is the universal quotient bundle on $\PP^2$. These can be regarded as building blocks for the matrices we are looking for, 
and in fact all our examples will be constructed from these blocks. This is the content of Section \ref{prel}.
We stress the fact that, even though the Chern classes of the bundles that we build belong to the stable range, not all those bundles are stable, 
as explained in Corollary \ref{cor stabili/instabili}.

The unstable range where $c^2-4y \ge 0$ is treated in Section \ref{unstable}; there we prove that the question has in general negative answer, by producing a series of counterexamples. 
The paper ends with a particularly interesting class of examples of effective pairs that are not $m$-effective, but that nevertheless 
induce an embedding of the projective plane $\PP^2$ in the Grassmannian of lines $\GG(1,2c+1)$. 

\vspace{0,2cm}

\noindent \textbf{Acknowledgements.} We would like to thank Philippe Ellia for sharing his ideas on globally generated vector bundles, and 
for suggesting the proof of Proposition \ref{c,2c}.

\section{Preliminaries}\label{prel}

We start by introducing a large class of $m$-effective bundles, obtained as ``quotients'' of certain decomposable bundles. We warn the reader that 
by quotient we mean something more specific than the usual definition, that is: 
\begin{defn}\label{def quoziente} A vector bundle $E$ on a projective space $\PP$ is a \emph{quotient of a vector bundle $F$} if there exist $s \ge 0$ sections of 
	$F$ inducing the exact sequence: 
	\begin{equation}\label{sec quoziente}
		0\ra \OO_{\PP}^{s} \ra F\ra E \ra 0.
	\end{equation}
\end{defn}

This same definition is used in \cite{sierra_ugaglia_gg} and \cite{Fania_Mezzetti}. Notice that if $E$ is a quotient of $F$, then their Chern classes 
satisfy $c_i(E)=c_i(F)$ for all $i\leq \rk E$.

\vskip.03in

As anticipated in the Introduction, we are interested in quotients of bundles of the form:
\begin{equation}\label{sommadiretta} 
	F=(\oplus_{i\geq 1}\OP(i)^{a_i})\oplus Q^b,
\end{equation} 
with $i\geq 1$ and $a_i, b\geq 0$ for all $i$, i.e. quotients of special decomposable bundles that are direct sums of $Q$ and the line bundles $\OP(i)$, $i\geq 1$. 
$Q$ is the universal quotient bundle on $\PP^2$, or, if one prefers, $Q=\TP(-1)$ is a twist of the tangent bundle. 
All $m$-effective bundles appearing in this paper are of this form. To our knowledge, there are no other examples. 

\begin{thm}\label{quot}
Let $E$ be a rank $2$ vector bundle on $\PP^2$, quotient of a direct sum of copies of $Q$, 
and of the line bundles $\OP(i)$, $i\geq 1$. Then $E$ is $m$-effective.
\end{thm}

Before proving Theorem \ref{quot}, we recall a few facts about the secant varieties of the Grassmannians of lines. 
Let $\GG(1,n-1)\subset \PP(\Lambda^2 V)$ denote the Grassmannian of the $2$-dimensional vector subspaces of  a vector space $V$ of dimension 
$n$, or, equivalently, of the lines of $\PP(V)$. 
We denote by $\sigma_k\GG(1,n-1)$ its $k$-th secant variety, $k\geq 1$, i.e. the Zariski closure of the union of the $(k-1)$-spaces generated by $k$ independent points of $\GG(1,n-1)$. 
It is well known that the points of $\sigma_k\GG(1,n-1)$ can be interpreted as skew-symmetric matrices of size $n$ and rank at most $2k$. Therefore, given a skew-symmetric matrix $A$ of linear forms in $3$ variables, of constant rank $2c$  and size $n$, it is natural to interpret its projectivization $\PP(A)$ as a (projective) $2$-plane contained in the 
stratum $\sigma_c\GG(1,n-1)\setminus \sigma_{c-1}\GG(1,n-1).$
 
The following result can be found in \cite{Fania_Mezzetti}; we give a version that is suitable for our purposes.
\begin{prop}\cite[Coroll. 5.9]{Fania_Mezzetti}\label{proj}
	Let $c$ be any positive integer, and let $A$ be a $3$-dimensional linear space of matrices of size $N\geq 2c+2$ and constant rank $2c$. 
	Then $\PP(A)$ can be isomorphically projected to $\sigma_{c}\GG(1,2c+1)\setminus \sigma_{c-1}\GG(1,2c+1)$.  
 \end{prop}
 
Therefore, from a matrix $A$ of size $N$ we can obtain, by projection, a new matrix $A'$ of size $2c+2$, whose rank remains constant and equal to $2c$. 
In \cite{Fania_Mezzetti} it is shown how one can choose the centre of projection to produce explicitly a projection of $\PP^{N-1}$ to $\PP^{2c+1}$ with the required property.
In particular \cite[Examples 5 and 6]{Fania_Mezzetti} provide explicit examples of such a projection.
 
\begin{rem}\label{proiezione}
	If $\coker A$ is a vector bundle $F$ of rank $N-2c$, performing this projection is equivalent to taking a rank $2$ quotient $E$ of $F$, such that $E=\coker A'$. Conversely, a general rank $2$ quotient $E$ of $F$ gives rise to a constant rank matrix $A'$. This can be seen by combining the two exact sequences \eqref{estensione2} and \eqref{sec quoziente} in a commutative 
	diagram, using diagram chase and the Snake Lemma. 
\end{rem}

To construct constant rank matrices of co-rank $2$ and of the desired size, we will therefore first build bundles of high rank, and then project to get 
quotients of rank $2$. These high rank vector bundles are all constructed from two types of building blocks, namely the universal quotient $Q$ 
and the line bundles $\OP(i)$, $i\geq 1$. 

We introduce these two types of building blocks in the following examples.

\begin{ex}\emph{The universal quotient $Q$.} The bundle $Q$ is $m$-effective. An extension of the form \eqref{estensione2} 
	for it can be constructed by taking the direct sum of the Euler sequence: 
\begin{equation} 0 \ra \OP(-1) \ra \OP^3 \ra Q \ra 0 \end{equation}
and of its dual twisted by $-1$. The resulting extension is therefore:
\begin{equation}\label{cotang}
  \xymatrix{0 \ar[r] & \Omega_{\PP^2} \ar[r] & \OP(-1)^{4}\ar[r]^-{A_Q} & \OP^4 \ar[r]&Q \ar[r]&0,}
\end{equation}
where  $A_Q$ is a $4\times 4$ matrix of constant rank $2$ of the form:
 \begin{eqnarray*}\begin{pmatrix}
            0 & a & b  & c  \\
           -a & 0 & 0  & 0  \\
           -b &  0 & 0  & 0  \\
           -c & 0  &  0  & 0  \\
       \end{pmatrix}.  \end{eqnarray*}
\end{ex}

\begin{ex}\emph{The line bundles $\OP(i)$.} For every odd number $2i+1\geq 3$  a general linear system of dimension 3 of  skew-symmetric matrices of size $2i+1$ is of constant rank $2i$, with extension 
\begin{equation}\label{Ai}
  \xymatrix{0 \ar[r] & \OP(-i-1) \ar[r] & \OP(-1)^{2i+1}\ar[r]^-{A_i} & \OP^{2i+1} \ar[r]&\OP(i) \ar[r]&0.}
\end{equation}
This simply follows by the fact that the secant variety $\sigma_i\GG(1,2i)$ has codimension $3$ in $\PP(\Lambda^2 V)$, where now $\dim V=2i+1$. An explicit example is the matrix 
 \begin{eqnarray*} A_i =\begin{pmatrix}
            0 & . & .  & . & 0 & a & b  \\
            . &    &    &   & a & b &  c   \\
            . &    &    &   & \iddots & c & 0 \\
              &    &    & \iddots  &   &   & \\
           0 & -a & \iddots  & & & & .  \\
           -a &  -b & -c  & & & & .  \\
           -b & -c  &  0  & . & . &  . & 0 \\
       \end{pmatrix}.  \end{eqnarray*}

\end{ex}
\medskip

We are now ready to prove Theorem \ref{quot}.
\begin{proof}[Proof of Theorem \ref{quot}] Assume that $E$ is a rank $2$ quotient of a bundle $F$ of the form \eqref{sommadiretta}. 
	Take the direct sum of $b$ copies of the matrix $A_Q$ and, for all $i$, 
	$a_i$ copies of the matrix $A_i$ and let $A$ be the direct sum of all these matrices: its Coker is $F$. To conclude it is enough to apply Proposition \ref{proj} and 
	Remark \ref{proiezione}.
\end{proof}

\begin{rem}\label{chern}
The Chern classes of the bundle $F$ appearing in \eqref{sommadiretta} can be computed using repeatedly the well-known formulas 
$c_1(A\oplus B)=c_1(A)+c_1(B)$ and $c_2(A\oplus B)=c_2(A)+c_2(B)+c_1(A)c_1(B)$, with $A$ and $B$ any two vector bundles on 
a projective space $\PP$. It is then straightforward (and boring) to see that:
\begin{align}
\label{c1} &c_1(F)=\sum_{i}ia_i+b,\\
\label{c2} &c_2(F)= \sum_i\  i^2{{a_i}\choose{2}}+\sum_{i\neq j} i j a_ia_j+{{b+1}\choose{2}}+b\:\Big{(}\sum_iia_i\Big{)}.
\end{align}
\end{rem}

\begin{rem} Let $E$ be a globally generated rank $2$ bundle on $\PP^2$. If $c_1(E)\leq 2$, then $E$ is a quotient of a bundle of the form \eqref{sommadiretta}, see \cite{sierra_ugaglia_gg}. 
	If $c_1(E)=3$, the same property holds true under the additional assumption that $E$ defines an embedding in $\GG(1,7)$, see \cite{Fania_Mezzetti}. 
 \end{rem}

The example below gives some information on the cases $c_1=4$ and $c_1=6$.
 
\begin{ex} Let $\EE$ be a (mathematical)  instanton bundle of charge $k$ on $\PP^3$, i.e. a rank $2$ vector bundle on $\PP^3$ defined as the cohomology of a linear monad 
	of type $\OO_{\PP^3}(-1)^k \to \OO_{\PP^3}^{2k+2} \to \OO_{\PP^3}(1)^k$. In \cite{bo_fa_me} the following two facts are proved:
 \begin{enumerate}
 \item  if $\EE$ is any charge $2$ instanton, then $E=\EE(2)$ is the cokernel of a skew-symmetric matrix of linear forms in $4$ variables, having size $10$ and constant rank $8$. 
 \item If $\EE$ is a general charge $4$ instanton, then $E=\EE(3)$ is the cokernel of a skew-symmetric matrix of linear forms in $4$ variables having size $14$ and constant rank $12$.  
 \end{enumerate}

It is clear that the restrictions of these bundles to $\PP^2$ are $m$-effective. 

Notice that the bundle $E$ on $\PP^3$ cannot be a quotient of a bundle of higher rank. Indeed, an exact sequence of type \eqref{sec quoziente} 
corresponds to an element of the group $\Ext^1(E,\OO_{\PP^3}^s) \simeq \HH^1(E^*)$. 
From the monad defining $\EE$ it is easy to compute cohomology and check that this $\HH^1$ group vanishes. 

The behavior of the restricted bundle $E|_{\PP^2}$ is quite different. The same cohomology computation gives us $h^1(E|_{\PP^2}^*)=k$, with 
$k$ the charge of the instanton. For the (restricted) instanton of charge 2 we can say a little more. The $m$-effective pair associated to $E$ is in this case 
$(4,6)$. In the recent work \cite{anghel_coanda_manolache} it is shown that the only possibilities for such a pair are to be associated to either a quotient of 
$\OP(1)^4$, or a quotient of $\OP(1) \oplus \OP(2) \oplus Q$. The first case corresponds to a stable Steiner bundle, while the second one is semistable. 
For the behavior of the charge $2$ instanton restricted to planes, see also \cite[Prop. 9.10]{Hart_ist}.
\end{ex}

\section{Bounds and necessary conditions}\label{bounds}

Let $(c,y)$ be an effective pair, and let $E$ be an associated globally generated vector bundle on $\PP^2$. From the fact that the
restriction of $E$ to a line is also globally generated, we deduce that $c=c_1(E) >0$. By taking a general section of $E$ 
and looking at its (smooth!) zero locus we also see that one must have $y=c_2(E) \ge 0$. 
It turns out that $m$-effectiveness imposes not just a lower, but also an upper bound on this second value $y$; this will be the content of Proposition 
\ref{sharp bounds}. One of the main ingredients of its proof consists in a necessary vanishing in cohomology 
that a vector bundle $E$ must satisfy in order to fit in a 2-step extension of type \eqref{estensione2}. 
This easy fact will come in handy in Section \ref{unstable} and thus deserves to be mentioned in the following:

\begin{lem}\label{h1 vanishing} 
The exact sequence \eqref{estensione2} entails that $\hh^1(E(-1))=\hh^2(E(-1))=0$. Therefore, the vanishing of these cohomology 
groups is a necessary condition for a  bundle $E$ to be $m$-effective.	
\end{lem}

\begin{proof}
	We compute the cohomology of the vector bundle $E(-1)$ using the (twisted) exact sequence \eqref{estensione2}. 
	Since the cohomology of $\OP(-2)$ and $\OP(-1)$ vanishes in all degrees, we deduce that $\hh^1(E(-1))=\hh^2(E(-1))=0$.
\end{proof}

\begin{prop}\label{sharp bounds}
  Let $c$ and $y$ be two non-negative integers. If the pair $(c,y)$ is $m$-effective, then $c>0$ and $y$ satisfies the sharp inequality 
 \mbox{$0 \le y \le {c+1 \choose 2}$}.
\end{prop}

\begin{proof} 
The fact that $c>0$ and $y \ge 0$ follows from effectiveness, as already remarked. On the other hand,  Lemma \ref{h1 vanishing} entails that 
the Euler characteristic $\chi(E(-1))$ equals to $\hh^0(E(-1))$, and therefore must be non-negative. 
Using Riemann-Roch we compute that $0 \le \chi(E(-1))=\frac{c(c+1)}{2}-c_2(E)$, which is equivalent to $y\leq {c+1 \choose 2}$.

For the sharpness part of the statement: the lower bound $y=0$ is attained by taking $E=\OP \op \OP(c)$, 
whereas if $E$ is a Steiner bundle defined by the short exact sequence:

\begin{equation}\label{steiner}
  0 \ra \OP(-1)^c \ra \OP^{c+2} \ra E \ra 0,
\end{equation}

\noindent then $c_2(E)=\frac{c(c+1)}{2}={c+1 \choose 2}$. 

To see that a Steiner bundle is indeed $m$-effective, notice that $E$ is a Steiner bundle defined by \eqref{steiner} if and only if it is a quotient of $Q^c$. 
Hence $m$-effectiveness follows from Theorem \ref{quot}. Indeed, one simply needs to look at the following commutative diagram, where all rows and columns are exact. 

$$\xymatrix{
&&0 \ar[d]&0 \ar[d]&\\
&&\OP^{2c-2} \ar[d]\ar@{=}[r]&\OP^{2c-2} \ar[d]&\\
0 \ar[r]&\OP(-1)^{c} \ar[r]\ar@{=}[d]&\OP^{3c}\ar[d]\ar[r]&Q^c\ar[d]\ar[r]&0\\
0 \ar[r]&\OP(-1)^{c} \ar[r]&\OP^{c+2}\ar[d]\ar[r]&E\ar[d]\ar[r]&0\\
&&0&0&}
$$
\noindent The central row is nothing but $c$ copies of the Euler sequence.
\end{proof}

\vskip.1in

We denote by $I$ the closed interval $I:=[0,{c+1 \choose 2}]$ that we obtain from Proposition \ref{sharp bounds}. A well-known result of 
Schwarzenberger entails that the Chern classes $(c,y)$ of a stable rank $2$ bundle on $\PP^2$ satisfy the inequality $\Delta:= c^2-4y<0$ (
and $\Delta \neq -4$). The interval $I$ on the other hand contains values for which $\Delta$ can be both negative and non-negative. 
This distinction prompts us to divide $I$ into two sub-intervals $I=I_u \cup I_s$. The sub-intervals \mbox{$I_s:=]\frac{c^2}{4},{c+1 \choose 2}]$} and 
\mbox{$I_u:=[0,\frac{c^2}{4}]$} will be called the \emph{stable} and (respectively) \emph{unstable range}.

\section{The stable range}\label{stable}

In this section we study the stable range: we show that in such range effectiveness and $m$-effectiveness coincide, or, in other words, 
that all pairs $(c,y)$ such that $y \in I_s$ are $m$-effective. Indeed, from \cite[Prop. 6.5]{le_potier_stabili} and \cite[Coroll. 1.5]{ellia_effective} 
we learn that all pairs $(c,y)$ such that $y \in I_s$ are effective. 

\begin{thm}\label{teo range stabile} 
	Let $c$ be a positive integer, and let \mbox{$I_s=]\frac{c^2}{4},{c+1 \choose 2}]$}. All pairs $(c,y)$, with $y$ any integer belonging to $I_s$, are $m$-effective. 
	Hence for all such pairs there exists a skew-symmetric matrix $A$ of linear forms in $3$ variables, having size $2c+2$, constant rank $2c$, cokernel $E$, and 
	$(c_1(E),c_2(E))=(c,y)$. Moreover such a matrix can be constructed explicitly. 	
\end{thm}
\begin{proof}
We will prove that, for any pair $(c,y)$ with $c>0$ and $y$ in the stable range $I_s$, there exists a rank $2$ bundle $E$, with $(c_1(E),c_2(E))=(c,y)$, quotient of a bundle $F$ of the form \eqref{sommadiretta} $F=(\oplus_{i\geq 1}\OP(i)^{a_i})\oplus Q^b$. Then it will be enough to apply Theorem \ref{quot} to conclude that $E$ is $m$-effective. Moreover, 
we will give an explicit algorithm to compute the bundle $F$.

First, divide $I_s$ into two subintervals: $I_s=I_{s1} \cup I_{s2}$, where $I_{s1}=]\frac{c^2}{4},{c\choose 2}[$ and \mbox{$I_{s2}:=[{c \choose 2},{c+1 \choose 2}]$}.

For the interval $I_{s2}$ it is enough to reformulate in terms of $m$-effectiveness the results contained in \cite[Section 6]{le_potier_stabili}: if $y\in I_{s2}$, 
there is an open subset of the  moduli space $M_{\PP^2}(2;c,y)$ whose general element is a rank $2$ bundle $E$ with $c_1(E)=c$ and $c_2(E)=y$, 
fitting in a short exact sequence of the form:
$$0 \ra \OP^{y-2-\frac{c(c-3)}{2}} \ra Q^{y-{{c}\choose{2}}} \op \OP(1)^{{{c+1}\choose{2}}-y} \ra E \ra 0.$$
In other words, $E$ is a quotient of 
a direct sum of copies of $Q$ and $\OP(1)$.
The computation of the Chern classes follows from Remark \ref{chern}.

\vskip.05in

For $I_{s1}$ things are much more complicated, and we thus proceed step by step.  

\vskip.05in

Fix $c>0$ and $y$ with $\frac{c^2}{4}<y<{c\choose 2}$. 

\emph{\underline{Step 1.}} It is convenient to introduce the constant $x={c\choose 2}-y,$ with $0<x<\frac{c^2-2c}{4}$. 
Note that if $y$ is of the form \eqref{c2}, i.e. $y=c_2(F)$ with $F$ as in \eqref{sommadiretta}, then it is easy to compute that: 
\begin{equation}\label{x} 
	\mbox{$x=\sum_{i\geq 2} \ a_i{i \choose 2}-b.$}
\end{equation}
We stress that the value $a_1$ does not appear in the expression of $x$. In the next steps we will look for a convenient expression of $x$ of the form \eqref{x}, 
suitable for our purposes.

\vskip.05in

\emph{\underline{Step 2.}} For $x>0$, set $c_m(x):=\min\{c\in\mathbb N\mid x<\frac{c^2-2c}{4}\}$. So $c_m(x)=1+\lceil \sqrt{4x+1}\rceil$, where 
for a real number $z$ we denote by $\lceil z\rceil$ the minimum integer number \emph{strictly} bigger than $z$. We warn the reader that 
$\lceil z\rceil$ coincides with the ceiling of $z$ only when $z$ is not an integer. 
Notice that, if there exists a bundle $F$ as in \eqref{sommadiretta} for the pair \mbox{$(c_m(x), {c_m(x)\choose 2}-x)$}, then there exists a bundle $F'$ as in \eqref{sommadiretta} for any pair $(c, {c\choose 2}-x)$ with $c\geq c_m(x)$.
Indeed it is enough to take $F'=F\oplus \OP(1)^{c-c_m(x)}$. 

Thus we focus our attention on the pairs $(c, y)$ of the form $(c_m(x), {c_m(x)\choose 2}-x)$. 

\vskip.05in

\emph{\underline{Step 3.}} So let $x \ge 0$ be an integer. 
If $x\leq 2$, we set $a_2:=x$. If instead $x>2$, let $k_1:=\max\{k>2\mid {k \choose 2}\leq x\}$. Then $x={k_1\choose 2}+z$, with $0\leq z\leq k_1-1$. 

Iterate this construction for $x$, starting from $z$: if $z\leq 2$, we set $a_2=z$, if instead $z>2$, we let $k_2=\max\{k>2\mid {k \choose 2}\leq z\}$. Keeping on 
repeating the same construction, one ends up with an expression:

\begin{equation}\label{decomposizione x}
\mbox{$x={k_1\choose 2}+{k_2\choose 2}+\ldots +{k_h\choose 2}+a_2,$}
\end{equation}

\noindent with uniquely determined $a_2\in \{0,1,2\}$, and, if $x>2$, $h\geq 1$, $k_1>k_2>\ldots> k_h>2$. 
Remark that $x-{k_1\choose 2}\leq k_1-1$, $x-{k_1\choose 2}-{k_2\choose 2}\leq k_2-1$, and so on.

\vskip.05in

\emph{\underline{Step 4.}} Define now the new value $\gamma_0(x):=k_1+k_2+\ldots +k_h+2a_2$. If $\gamma_0(x)\leq c_m(x)$, then we can define a bundle $F$ with $c_1(F)=c_m(x)$ by setting $F=\OP(k_1)\oplus \OP(k_2)\oplus\ldots\oplus\OP(2)^{a_2}\oplus \OP(1)^\alpha$, with $\alpha=c_m(x)-\gamma_0(x)$. From \eqref{x} we get that $c_2(F)={c_m(x)\choose 2}-x$, hence 
in this case we are done.

\vskip.05in

\emph{\underline{Step 5.}} If instead $\gamma_0(x) > c_m(x)$, then for any $1\leq i\leq h$ we consider the expressions $x= {k_1\choose 2}+\ldots +{k_{i-1}\choose 2}+{k_i+1\choose 2}-b_i$,  which define the numbers $b_i$ with $1\leq b_i\leq k_i$. Set $\gamma_i(x):=k_1+k_2+\ldots +k_{i-1}+(k_i+1)+ b_i$, for $i >0$. Then, as soon as $\gamma_i(x)\leq c_m(x)$ for some $i=1,\ldots,h$, we can reduce to the previous 
case and take the following bundle with the desired Chern classes: $F_i=\OP(k_1)\oplus \ldots \oplus \OP(k_{i-1})\oplus \OP(k_i+1)\oplus Q^{b_i}\oplus \OP(1)^{\alpha_i},$ with $\alpha_i=c_m(x)-\gamma_i(x)$.

To conclude the proof we only need to show that this is always the case, i.e. that every time we repeat the construction above we can indeed find such an $i$ with 
$\gamma_i(x)\leq c_m(x)$. 

\vskip.1in
\textbf{Claim.} \emph{For any $x>0$ there exists $i\geq 0$ such that $\gamma_i(x)\leq c_m(x)$.}
\vskip.1in

If $x=1,2$ the claim is clearly true.

If $x\geq 3$, we use induction on $k_1$. Let us now denote $k_1$ by $k_1(x)$, in order to 
underline its dependence on $x$. We will check that the claim is true for the numbers $x$ having low $k_1(x)$, where the term \lq\lq low'' will be made precise in a moment. 
For the inductive step we observe that, if $x={{k_1(x)}\choose 2}+z$, with $z\leq k_1(x)-1$, then $k_1(z)<k_1(x)$ and $\gamma_i(x)=k_1(x)+\gamma_i(z)$. 
(This holds for $i=0$ and for any $i\geq1$ such that both $\gamma_i(x)$ and $\gamma_i(z)$ make sense.) 
Therefore, assuming that the claim is true for $z$, we want to  deduce that it is true for $x$. It is enough to prove that $k_1(x)+c_m(z)\leq c_m(x)$. From the next Lemma it follows that this is true for $k_1(x)\geq 25$. Hence the first values of $k_1(x)$ to check preliminarily are $k_1(x)\leq 24$.

\begin{lem}
Assume that  $x={{k_1(x)}\choose 2}+z$ with $z\leq k_1(x)-1$, and that moreover $k_1(x)\geq25$. Then:
\begin{equation}\label{cm}
\mbox{$k_1(x)+c_m(z)\leq c_m({{k_1(x)}\choose {2}} +z).$}
\end{equation}
\end{lem}
\begin{proof}
This amounts to verify the following inequality  (where for simplicity we write $k$ instead of $k_1(x)$): 
\begin{equation}
\mbox{$k+\left\lceil\sqrt{4z+1}\right\rceil  \leq \left\lceil {\sqrt{4{k \choose 2}+4z+1}}\right\rceil = \left\lceil {\sqrt{2k^2-2k+4z+1}}\right\rceil.$}
\end{equation} 
Clearly it is enough to show that:
\begin{equation}
k+1+\sqrt{4z+1}  \leq   {\sqrt{2k^2-2k+4z+1}},
\end{equation} 
and this reduces to:
\begin{equation}
\mbox{$z\leq \frac{k^4-8k^3+10k^2-3}{16(k+1)^2}=\frac{(k-1)^2(k^2-6k-3)}{16(k+1)^2}.$}
\end{equation} 
Since $z\leq k-1$, it is enough to prove:
\begin{equation}\label{terzogrado}
\mbox{$k-1\leq \frac{(k-1)^2(k^2-6k-3)}{16(k+1)^2}.$}
\end{equation} 
It is easy to check that inequality \eqref{terzogrado} is satisfied for $k\geq 25$.
\end{proof} 
A brute force computation shows that the statement is true for all values of $x$ having $k_1(x)\leq 24$, and this concludes Step 5 as well as the proof 
of Theorem \ref{teo range stabile}
\end{proof}

\begin{rem}
	For more details on the explicit computation and on the techniques used, we refer to Section \ref{esempio}, where we work out 
	the example $c=8$.
\end{rem}

Even if their Chern classes belong to the stable range, not all the $m$-effective vector bundles that we build are stable. 
Recall that a rank $2$ vector bundle on $\PP^2$ is stable if and only if its normalized bundle (i.e. the twist of the bundle with first Chern class equal to 
$0$ or $-1$) has no sections. Using this we can prove the following:

\begin{prop}
Let $E$ be a rank $2$ quotient of the bundle $F=\oplus_{i\geq 1}\OP(i)^{a_i}\oplus Q^b$, with $c_1(E)=c$. Set $\iota:=\max\{i\mid a_i>0\}$. $E$ is stable if and only if
$\iota <c/2$.
\end{prop}
\begin{proof} 
It is enough to note that the normalized bundle of $E$ is $E(-\nu)$,  with $\nu= \frac{c}{2}$ if $c$ is even and $\nu=\frac{c+1}{2}$ if $c$ is odd. Therefore $E$ is stable if and only if $i-\nu<0$ for any $i$ such that $a_i\neq 0$.
\end{proof}
 \begin{cor} \label{cor stabili/instabili} If $c$ is even and $y\geq {\frac{c^2}{4}}+c-3$ (resp. if $c$ is odd and $y\geq{\frac{c^2}{4}}+{\frac{2c-3}{4}}$), there exist  $m$-effective stable bundles.
  \end{cor}
 \begin{proof}
 Assume that $c$ is even. From \ref{c2} it follows that the minimal $c_2$ for bundles $E$  as in Theorem \ref{teo range stabile}, under the condition $\iota <c/2$, is attained for $b=0$ when the number of the indices $i$ such that $a_i>0$ is the minimum possible, i.e. $3$. So we consider 
 $c_2(\OP(\frac{c}{2}-\alpha)\oplus \OP(\frac{c}{2}-\beta)\oplus \OP(\alpha +\beta))$. 
This is a function of $\alpha$ and $\beta$ that, for $0\leq \alpha, \beta<\frac{c}{2}$, attains its minimum for $(\alpha, \beta)=(1,1)$. For $c$ odd, a similar argument applies. \end{proof}

\section{The unstable range}\label{unstable}

We now consider the unstable range, that is, the interval $I_u=[0,\frac{c^2}{4}]$. As a first remark, notice that both endpoints correspond 
to $m$-effective pairs. Indeed, if $c$ is even and $\frac{c^2}{4} \in \N$, then the two pairs $(c,0)$ and $(c,\frac{c^2}{4})$ are attained by the 
two quotient bundles $\OP \oplus \OP(c)$ and $\OP(\frac{c}{2})^2$ respectively, and their $m$-effectiveness follows from Theorem \ref{quot}. 
If instead $c$ is odd, then the right-endpoint of $I_u$ is $\frac{c^2-1}{4}$, and the pair $(c,\frac{c^2-1}{4})$ corresponds to $\OP(\frac{c-1}{2})\oplus\OP(\frac{c+1}{2})$.

\vskip.1in

Recall that effectiveness is a necessary condition for $m$-effectiveness. We are thus interested in studying effective 
pairs $(c,y)$ with $y \in I_u$. These have been completely classified in \cite{ellia_effective}. Contrary to what happens in the stable 
range, where, given $c$, all values $y \in I_s$ give an effective pair $(c,y)$, in the unstable range there are values $y \in I_u$ such that 
the pair $(c,y)$ cannot be attained as Chern classes of a globally generated bundle on $\PP^2$. In other words, there are \emph{gaps} in the effective range. 
The description of these gaps is quite involved; we report it for the reader's convenience, with the warning that our notation 
is slightly different from the original one. 

Denote by $\lf \frac{c}{2} \rf$ the integral part of $\frac{c}{2}$. For every integer $0 \le k \le \lf\frac{c}{2}\rf-1$, let $G_k(0):=[kc + 1,(k+1)c-(k+1)^2-1]$, 
with the convention that if $b>a$ then $[a,b]=\emptyset$. For $3 \le k \le \lf\frac{c}{2}\rf-1$, set $k_0:=\lf\sqrt{k-2}\rf$. 
For every integer $a$ such that $1 \le a \le k_0$ define $G_k(a):=[k(c-a)+a^2+1, k(c-a)+k-1]$. 

Finally, let $G_k=\cup_{a=0}^{k_0} G_k(a)$ and $G=\cup_{k=0}^{\lf\frac{c}{2}\rf-1} G_k$. 
Ellia's classification entails that if $y \in I_u$, the pair $(c,y)$ is effective if and only if $y \in \tilde{I}_u:=I_u \setminus G$. 

\vskip.1in
What can we say about the $m$-effectiveness of pairs $(c,y)$, with $y \in \tilde{I}_u$? 

Define, for all $0 \le k < \lf\frac{c}{2}\rf$, the intervals:

\begin{equation}\label{J_k}
	J_k:=[k c -k^2,(k+1) c -(k+1)^2-1].
\end{equation}

The unstable range $I_u=[0,\frac{c^2}{4}]$ is subdivided into the $\lf\frac{c}{2}\rf$ subintervals $J_k$'s, each of length $c-2k-2$, plus the last endpoint $\lf\frac{c^2}{4}\rf$.

Notice that for all $0 \le k < \lf\frac{c}{2}\rf$ there exist $m$-effective pairs $(c,y)$ with $y \in J_k$. It is enough to consider quotient bundles of type 
$\OP(k) \op \OP(c -k)$, whose Chern classes are $c_1=c$ and $c_2= k c -k^2$. (The two endpoints of $I_u$ thus correspond to the two special values $k=0$ and $k=\frac{c}{2}$ in the 
even case, and $k=0$ and $k=\frac{c-1}{2}$ in the odd case.)

On the other hand, the intervals $J_k$'s contain gaps where the pair $(c,y)$ is not effective. Indeed for all $k$, the set of gaps 
$G_k$ is a subset of the interval $J_k$. We call $\tilde{J}_k:=J_k \setminus G_k$. 

The following result sheds some light on the structure of the intervals $\tilde{J}_k$'s.

\begin{prop}\label{instabile generale} Let $c$ be any positive integer, and let $I_u=[0,\frac{c^2}{4}]$ be the unstable range. For $0 \le k < \lf\frac{c}{2}\rf$ let 
	$J_k=[k c -k^2,(k+1) c -(k+1)^2-1]$, so that $I_u=\cup_{k=0}^{\lf\frac{c}{2}\rf-1} J_k \cup \{\frac{c^2}{4}\}$ if $c$ is even, and 
	$\cup_{k=0}^{\lf\frac{c}{2}\rf-1} J_k \cup \{\frac{c^2-1}{4}\}$ if $c$ is odd. Consider a pair $(c,y)$, with $y$ any integer belonging to $I_u$. 
	Then:
	\begin{enumerate} 
		\item If $y \in J_0$ or $y \in J_1$, then $(c,y)$ is $m$-effective if and only if it is effective. 
		\item If $y \in J_k$ with $k \ge 2$, then there is a subset: 
		\[\mbox{$N_k:=[kc - k^2+ {k+1 \choose 2}+1,(k+1)c - (k+1)^2  - 1] \cap \tilde{J}_k$}\] 
		such that, as soon as $c > (k+1)^2$, the pair $(c,y)$ with $y \in N_k$, is effective, but not $m$-effective.
	\end{enumerate}
\end{prop}

\begin{proof} The proof of the first part reduces to an easy remark. One has that $J_0=\{0\} \cup G_0(0)$, hence $\tilde{J}_0=\{0\}$, and 
	we have already seen more than once that the pair $(c,0)$ is attained by the rank $2$ bundle $\OP \oplus \OP(c)$. 
	Similarly, $J_1=\{c-1,c\} \cup G_1(0)$, hence $\tilde{J}_1=\{c-1,c\}$. The two values correspond to quotients of the bundles $\OP(1) \oplus \OP(c-1)$, and 
	$Q \oplus \OP(c-1)$ respectively. All these bundles are $m$-effective thanks to Theorem \ref{quot}.
	
	\vskip.1in
	
	For the second part, let $(c,y)$ be an effective pair, with $y \in \tilde{J}_k$, and let $E$ be the associated globally generated rank $2$ vector bundle. 
	We are interested in computing the group $\HH^1(E(-1))$. If we find a range in which this group is non-zero, 
	then by Lemma \ref{h1 vanishing} the bundle $E$ cannot be $m$-effective.
	
	Let $s$ be a global section of $E$. Via Hartshorne-Serre correspondence, we get a short exact sequence of type:
	\begin{equation}\label{sec def Y}
	  0 \ra \OP \ra E \ra \II_Y(c) \ra 0,
	\end{equation} 
	where the zero locus $Y=(s)_0$ is a locally complete intersection (l.c.i from now on) $0$-dimensional subscheme of $\PP^2$, of length $y$.  
	We will use properties of $Y$ to deduce information on the cohomology of the bundle $E$. Indeed, from the long exact cohomology sequence 
	induced by \eqref{sec def Y} twisted by $\OP(-1)$, we obtain the equality $\hh^1(E(-1))=\hh^1(\II_Y(c-1))$.
	
	As proven in \cite[Prop. 1.33]{g-h_residues}, the scheme $Y$ satisfies Cayley-Bacharach property $CB(c-3)$.

    We say that a l.c.i. $0$-dimensional subscheme $Y$ in $\PP^2$ satisfies the \emph{Cayley-Bacharach property} for curves of degree $n \ge 1$ if any curve of degree $n$ 
    containing a subscheme $Y' \subset Y$ of co-length $1$, contains $Y$. If this is the case, we write $Y$ satisfies $CB(n)$. 
    Remark that this implies that $Y$ satisfies $CB(i)$ for all $i \le n$.

    By \cite[Lem. 3.2]{ellia_effective}, since the group of sections $\HH^0(\II_Y(k))$ is non-zero, $Y$ lies on a curve of degree $k$, but not on a curve of degree $k-1$. 
    This allows us to obtain useful information on its numerical character, which in turn gives a method to compute $\hh^1(\II_Y(c-1))$.

	Recall that if $Z$ is a $0$-dimensional scheme in the projective plane, its \emph{numerical character} $\chi(Z)=(n_0,\ldots,n_{\sigma-1})$ is a sequence of integers 
	which encodes the Hilbert function of $Z$, with the following properties:
	\begin{enumerate}
	\item $n_0 \ge n_1 \ge \ldots \ge n_{\sigma -1} \ge \sigma$, where $\sigma$ is the minimal degree of a plane curve containing the scheme $Z$;
	\item $\deg Z = \sum_{i=0}^{\sigma -1} (n_i-i)$;
	\item $\hh^1(\II_Z(t))=\sum_{i=0}^{\sigma-1}[n_i-t-1]_+ - [i-t-1]_+$, where $[x]_+:=\max\{x,0\}$.
	\end{enumerate}
	The numerical character is \emph{connected} if $n_i \ge n_{i+1}+1$ for all $0 \le i \le \sigma-1$.

	In our setting, we have $\chi(Y)=(n_0,n_1,\ldots,n_{k-2})$ with: 
	\begin{enumerate}
	\item $n_0 \ge \ldots \ge n_{k-1} \ge k$, 
	\item $\sum_{i=0}^{k-1}n_i=y+{k \choose 2}$. 
	\end{enumerate}

	From \cite[Lem. 4.13]{ellia_effective} we learn that $\chi(Y)$ must be connected. 
	If for some index $j$ we had $n_{j}>n_{j-1}+1$ then $Y$ would not satisfy $CB(i)$ for all $i \ge n_r -1 \ge k-2$, where the second inequality follows from (a). 
	Since we do know that $Y$ satisfies $CB(c-3)$ and $k-2 \le \lf\frac{c}{2}\rf-2 < c-3$, this cannot happen.

\vskip.05in

 	We can thus use the numerical character to compute that: 
	$$\hh^1(\II_Y(c-1))= \sum_{i=0}^{k-1} [n_i -c]_+ - [i-c]_+=\sum_{i=0}^{k-1} [n_i -c]_+,$$
	\noindent where the last equality holds because $i \le k-1 <c$, and thus $[i-c]_+=0$ for all $i$. 

	Write the value $y \in \tilde{J}_k$ as $y=kc-k^2+\alpha$, with $0\le \alpha \le k^2$. Then we have:
	$$\sum_{i=0}^{k-1}n_i=y+{k \choose 2}=kc-k^2+\alpha + {k \choose 2}= kc +\alpha -{k+1 \choose 2}.$$
	The condition above entails that $n_0 \ge c-\frac{k+1}{2}+\frac{\alpha}{k}$, and thus $n_0-c \ge -\frac{k+1}{2}+\frac{\alpha}{k}$. 

	If $\alpha > {k+1 \choose 2}$ then $n_0 - c \gneq 0$, so the cohomology group $\HH^1(\II_Y(c-1))$ has positive dimension, and the pair $(c,y)$ cannot be $m$-effective.
	
	Before we can conclude, we need to make sure that the set $N_k$ is contained in the interval $J_k$, and this happens exactly as soon as $c > (k+1)^2$.
\end{proof}

Theorem \ref{instabile generale} combined with an explicit construction of $m$-effective bundles as quotients proves the following 
Corollary \ref{instabile fino 8}. Indeed, for $k\le 2$ one has that $\tilde{J}_k \setminus N_k= \emptyset$, meaning that 
for $c \le 7$ there is nothing to prove. For $c=8$ we refer the reader to Section \ref{esempio}, where this example is worked out in detail. 
Notice also that for low values of $c$ some of the $J_k$ intervals overlap, allowing us to construct several non-isomorphic bundles associated to the same 
$m$-effective pair.

\begin{cor}\label{instabile fino 8} 
	Let $0 < c \le 8$ be a positive integer, and let $I_u=[0,\frac{c^2}{4}]$ be the unstable range. 
	There is a complete classification of all $m$-effective pairs $(c,y)$ with $y \in I_u$.
\end{cor}

\vskip.2in

We recall that giving a globally generated rank $2$ vector bundle $F$ on $\PP^2$, together with a linear subspace $V \subseteq \HH^0(F)$ 
of dimension $N+1$, and an epimorphism $V \otimes \OP \epi F$, is equivalent to giving a map 
$\varphi_V:\PP^2 \to \GG(1,N)$ from $\PP^2$ to the Grassmannian of lines in $\PP^N$. When $V=\HH^0(F)$, we write $\varphi_F:=\varphi_{\HH^0(F)}$. 

Given an effective pair $(c,y)$ it is thus very natural to ask whether or not the associated globally generated bundle gives rise to an embedding. If 
the pair $(c,y)$ is also $m$-effective, then the answer to this question is always positive: 
for any bundle $E$ coming from an extension of type \eqref{estensione2}, the map $\varphi_F$ is an embedding of $\PP^2$ in $\GG(1,2c+1)$. 
This is proved in:

\begin{prop}\cite[Prop. 2.4]{Fania_Mezzetti} 
	Let $c$ be a positive integer, and let $A$ be a skew-symmetric matrix of linear forms in $3$ variables, 
	having size $2c+2$ and constant rank $2c$. Let $E$ be the globally generated vector bundle defined as the cokernel of $A$. 
	Then $E$ defines a $2c$-uple embedding of $\PP^2$ in $\GG(1,2c+1)$.
\end{prop}

From this viewpoint, the pair $(c,2c)$ is particularly interesting. In Proposition \ref{instabile generale} we have seen that 
as soon as $c > 9$, then $(c,2c)$ is not $m$-effective. Nevertheless, under some extra assumption the associated globally generated bundle $E$ 
gives an embedding of $\PP^2$ in the Grassmannian $\GG(1,2c+1)$. This is the content of the following:

\begin{prop}\label{c,2c}
	Given any integer $c > 9$, the pair $(c,2c)$ is effective but not $m$-effective, that is, there exists a 
	globally generated rank $2$ vector bundle $E$ on $\PP^2$ with Chern classes $(c,2c)$, but $E$ cannot be the cokernel of 
	a skew-symmetric matrix of linear forms in $3$ variables, having size $2c+2$ and constant rank $2c$. Moreover, 
	if there are no lines $L$ such that $E$ splits as $E|_L \simeq \OO_L \op \OO_L(c)$, 
	then $E$ induces an embedding $\varphi_V:\PP^2 \ra \GG(1,2c+1)$, where $V \simeq \CC^{2c+2} \subseteq \HH^0(E)$.
\end{prop}

\begin{proof}
The first part of the statement is a straightforward consequence of Proposition \ref{instabile generale}. For the second part, 
let $E$ be a globally generated vector bundle associated with the effective pair $(c,2c)$, and let us look at the induced map 
$\varphi_E$. We will start by proving that $\varphi_E$ is an embedding. 

Let $\xi \subset \PP^2$ be any 0-dimensional length $2$ subscheme. 
We need to show that the inequality $\hh^0(E \otimes \II_\xi) \le \hh^0(E)-3$ holds. Let $L$ be the line spanned by $\xi$. We have a sequence:
\begin{equation}\label{1}
  0 \ra \OP(-1) \ra \II_\xi \ra \OO_L(-2) \ra 0.
\end{equation}
Tensoring it by $E$ and computing cohomology, we see that:
\begin{equation}
  \hh^0(E \otimes \II_\xi)=\hh^1(E \otimes \II_\xi)+\hh^0(E(-1)) -\hh^1(E(-1)) + \hh^0(E|_L(-2)) - \hh^1(E|_L(-2)).
\end{equation}
On the other hand from the short exact sequence of definition of the hyperplane $L$ tensored by $E$ we get:
\begin{equation}
  \hh^0(E)= \hh^0(E(-1)) -\hh^1(E(-1))+ \hh^0(E|_L).
\end{equation}

Repeating step by step the proof of Proposition \ref{instabile generale}, one sees that for $k=2$ 
there is only one possibility for the numerical character of the 0-dimensional scheme $Y$ defined by \eqref{sec def Y}, namely 
$\chi(Y)=(c+1,c)$. Hence we can compute that $\hh^1(E(-1))=\hh^1(\II_Y(c-1))=1$.

Adding this information to what we knew before, we obtain:
\begin{equation}\label{2}
  \hh^0(E)-\hh^0(E \otimes \II_\xi)=\hh^0(E|_L) - [\hh^0(E|_L(-2)) - \hh^1(E|_L(-2))] -\hh^1(E \otimes \II_\xi).
\end{equation}

Since $E$ is globally generated, $E|_L \simeq \OO_L(c-b)\op \OO_L(b)$, where $\lf\frac{c}{2}\rf \le b\le c$. 
As long as $b \neq c$, we have $\hh^0(E|_L)-\hh^0(E|_L(-2))=4$ and $\hh^1(E|_L(-2))=0$, forcing $\hh^1(E \otimes \II_\xi) \le 1$. 
As it is shown in \cite[Rem. 2.5]{note_grass_arrondo}, this is a sufficient condition for $\varphi_E$ to be an embedding.

Notice that if there exists a line $L$ such that $E|_L \simeq \OO_L \op \OO_L(c)$, then $\hh^0(E|_L)=c$, $\hh^0(E|_L(-2))=c-1$ and $\hh^1(E|_L(-2))=1$, 
so \eqref{2} becomes:
\begin{equation}
  \hh^0(E)-\hh^0(E \otimes \II_\xi)=c-(c-1)+1 - \hh^1(E \otimes \II_\xi)=2- \hh^1(E \otimes \II_\xi)\le 2,
\end{equation}
and $\varphi_E$ cannot be an embedding.

To conclude our proof, it is enough to observe that the embedding $\varphi_E:\PP^2 \ra \GG(1,N)$, 
with $N=\hh^0(E)-1$, composed with the projection $\GG(1,N) \ra \GG(1,2c+1)$ is still an embedding as long as we stay out of the (5-dimensional) secant variety. 
From the cohomology of \eqref{sec def Y} we see that $N \ge \chi(E)-1 = \lf\frac{c}{2}\rf(c-1)+2 \ge 34$ if $c \ge 9$, 
so in our range this is always possible.
\end{proof}

\section{An example}\label{esempio}

Here we analyze in detail the case $c=8$. We believe that in this case the situation is simple enough to be explained 
in detail, yet complicated enough to have some interest for explaining our techniques. 

\vskip.05in

If $c=8$ then by Proposition \ref{sharp bounds} the value $y$ belongs to the interval $I=[0,36]$. 
$I$ is divided into the unstable range $I_u=[0,16]$ and stable range $I_s=[17,36]$, which in turn is 
$I_s=I_{s1} \cup I_{s2}=[17,27] \cup [28,36]$. 

\vskip.05in

The first table below lists all $m$-effective pairs $(8,y)$ with $y \in I_u$. Recall from Section \ref{unstable} that the unstable range $I_u$ is a 
union of sub-intervals $J_k$, with $k$ varying from $0$ to $3$, together with the value $\frac{c^2}{4}=16$. 
The intervals $J_k$'s are listed in the first column. The second column contains 
the values of $y$, and for each of them the third column contains either the explicit construction for the bundle $E$, in case the pair is 
$m$-effective, or, otherwise, the reason why $E$ cannot be constructed. Notice that in the case $y=15$ there are two non-isomorphic bundles that can 
be associated to the same $m$-effective pair $(8,15)$. 

\vskip.2in

\begin{center}
\bgroup
\def\arraystretch{1.3}
\begin{tabular}{|c|c|c|}
\hline
&$y=c_2(E)$& $E$ quotient of: \\
\hline
\hline
$J_0=[0,6]$&0&$\OP \op \OP(8)$\\
\hline
&$[1,6]$ & gap $G_0(0)$\\
\hline
$J_1=[7,11]$&7 & $\OP(1) \op \OP(7)$\\
\hline
&8 & $Q \op \OP(7)$\\
\hline
&$[9,11]$&gap $G_1(0)$\\
\hline
$J_2=[12,14]$&12 & $\OP(2) \op \OP(6)$\\
\hline
&13 &$\OP(1)^2 \op \OP(6)$\\
\hline
&14 &$Q \op \OP(1) \op \OP(6)$\\
\hline
$J_3=\{15\}$&15&  $\OP(3) \op \OP(5)$ or $Q^2 \op \OP(6)$\\
\hline
$\frac{c^2}{4}$&16& $\OP(4)^2$\\
\hline
\end{tabular}
\egroup
\end{center}

\vskip.1in

Let us now move on to the stable range; we start with $I_{s2}=[28,36]$. As explained in \cite{le_potier_stabili}, and in the proof of Theorem \ref{teo range stabile}, 
all these values are attained by bundles that are quotients of sums of copies of $\OP(1)$ and $Q$. 
More in detail, notice that if we have $c_2(\OP(1)^a \oplus Q^b)=y$ then $c_2(\OP(1)^{a-1} \oplus Q^{b+1})=y+1$, i.e. when we substitute 
a summand of type $\OP(1)$ with one of type $Q$ the second Chern class grows by 1. 
The endpoints $28$ and $36$ of the interval $I_{s2}$ are attained by (quotients of) $\OP(1)^8$ and $Q^8$ respectively, and all the intermediate values 
are attained by substituting step by step copies of $\OP(1)$ with copies of $Q$.

\vskip.1in

Finally, we look at the interval $I_{s1}=[17,27]$. The second table illustrates 
the algorithm in the proof of Theorem \ref{teo range stabile}, that allows one to explicitly construct 
the $m$-effective bundle associated to any pair $(c,y)$ with $y \in I_{s1}$. In the first three columns 
there are values of $y$, $x= {8 \choose 2}-y$, and of the decomposition \eqref{decomposizione x} of the latter, that we recall is of the 
form \mbox{$x={k_1\choose 2}+{k_2\choose 2}+\ldots +{k_h\choose 2}+a_2$}. (See Steps 1 and 3 of the algorithm.)

In the fourth column we wrote the value $\gamma_0(x)=k_1+k_2+\ldots +k_h+2a_2$ from Step 4, with uniquely determined $a_2\in \{0,1,2\}$. 
Following Step 2, we have added in the next column the value $c_m(x)=1+\lceil \sqrt{4x+1}\rceil$. 

Then in all cases where $\gamma_0(x) \le c_m(x)$ we can directly construct a bundle whose quotient is the $m$-effective $E$ we are after. 
The two cases where this does not happen are indicated with $(\ast)$ and $(\ast \ast)$. They correspond to cases where the algorithm stops at Step 4. 

\vskip.2in

\begin{center}
\bgroup
\def\arraystretch{1.4}
\begin{tabular}{|c|c|c|c|c|c|}
	\hline
	$y=c_2(E)$& $x={8 \choose 2}-y$&decomposition \eqref{decomposizione x}&$\gamma_0(x)$&$c_m(x)$&$E$ quotient of:\\
	\hline
	\hline
	17&11&${5 \choose 2}$ +1&7&8&$\OP(5) \op \OP(2) \op \OP(1)$\\
	\hline
	18&10&${5 \choose 2}$&5&8&$\OP(5) \op \OP(1)^3$\\
	\hline
	19&9&${4 \choose 2} + {3 \choose 2}$&7&8&$\OP(4) \op \OP(3) \op \OP(2)$\\
	\hline
	20&8&${4 \choose 2} +2$&8&7&$(\ast)$\\
	\hline
	21&7&${4 \choose 2} +1$&6&7&$\OP(4) \op \OP(2) \op \OP(1)^2$\\
	\hline
	22&6&${4 \choose 2}$&4&6&$\OP(4) \op \OP(1)^4$\\
	\hline
	23&5&${3 \choose 2}+2$&7&6&$(\ast \ast)$\\
	\hline
	24&4&${3 \choose2} +1$&5&6&$\OP(3) \op \OP(2) \op \OP(1)^3$\\
	\hline
	25&3&${3 \choose 2}$&3&5&$\OP(3) \op \OP(1)^5$\\
	\hline
	26&2&2 &4&4&$\OP(2)^2 \op \OP(1)^4$\\
	\hline
	27&1& 1&2&4&$ \OP(2) \op \OP(1)^6$\\
	\hline
\end{tabular}
\egroup
\end{center}

\vskip.1in

For the two cases marked with asterisks, we have to use Step 5 in the algorithm of Theorem \ref{teo range stabile}. For the case $x=8$, one has 
$\gamma_0(8)=8 > c_m(8)=7$. So let us write $8={5 \choose 2}-2$, and consider $\gamma_1(8)=7 =c_m(8)$. By applying the algorithm, we see that 
$E$ is a quotient of a bundle of the form $\OP(5) \op \OP(1) \op Q^2$. 

Similarly for $x=5$ we have $\gamma_0(5)=7 > 6=c_m(5)$. Then one rewrites $5$ as ${4 \choose 2}-1$, so that $\gamma_1(5)=5 < c_m(5)$ and $E$ is 
quotient of the direct sum $\OP(4) \op \OP(1)^3 \op Q$.

\vskip.1in

\bibliographystyle{amsalpha}
\bibliography{biblioada}

\def\cprime{$'$}
\providecommand{\bysame}{\leavevmode\hbox to3em{\hrulefill}\thinspace}
\providecommand{\MR}{\relax\ifhmode\unskip\space\fi MR }
\providecommand{\MRhref}[2]{%
  \href{http://www.ams.org/mathscinet-getitem?mr=#1}{#2}
}
\providecommand{\href}[2]{#2}
\begin{thebibliography}{ACM13}

\bibitem[ACM13]{anghel_coanda_manolache}
C.~Anghel, I.~Coanda, and N.~Manolache, \emph{Globally {G}enerated {V}ector
  {B}undles on {$\Bbb{P}^n$} with $c_1=4$}, arXiv:1305.3464, 2013.

\bibitem[Arr96]{note_grass_arrondo}
E.~Arrondo, \emph{Subvarieties of {G}rassmannians}, Lecture Note Series Dip. di
  Matematica Univ. Trento, no.~10, 1996,
  http://www.mat.ucm.es/~arrondo/trento.pdf.

\bibitem[BFM13]{bo_fa_me}
A.~Boralevi, D.~Faenzi, and E.~Mezzetti, \emph{Linear spaces of matrices of
  constant rank and instanton bundles}, Adv. Math. \textbf{248} (2013),
  895--920.

\bibitem[CP07]{pirola}
A.~Causin and G.~P. Pirola, \emph{A note on spaces of symmetric matrices},
  Linear Algebra Appl. \textbf{426} (2007), no.~2-3, 533--539.

\bibitem[DPM05]{dp_m}
P.~De~Poi and E.~Mezzetti, \emph{Linear congruences and hyperbolic systems of
  conservation laws}, Projective varieties with unexpected properties, Walter
  de Gruyter GmbH \& Co. KG, Berlin, 2005, pp.~209--230.

\bibitem[Ell13]{ellia_effective}
Ph. Ellia, \emph{Chern classes of rank two globally generated vector bundles on
  {$\Bbb{P}^2$}}, Atti Accad. Naz. Lincei Cl. Sci. Fis. Mat. Natur. Rend.
  Lincei (9) Mat. Appl. \textbf{24} (2013), no.~2, 147--163.

\bibitem[FM11]{Fania_Mezzetti}
M.L. Fania and E.~Mezzetti, \emph{Vector spaces of skew-symmetric matrices of
  constant rank}, Linear Algebra Appl. \textbf{434} (2011), 2383--2403.

\bibitem[GH78]{g-h_residues}
Ph. Griffiths and J.~Harris, \emph{Residues and zero-cycles on algebraic
  varieties}, Ann. of Math. (2) \textbf{108} (1978), no.~3, 461--505.

\bibitem[Har78]{Hart_ist}
R.~Hartshorne, \emph{Stable vector bundles of rank {$2$} on {${\Bbb P^{3}}$}},
  Math. Ann. \textbf{238} (1978), no.~3, 229--280.

\bibitem[IL99]{Ilic_JM}
B.~Ilic and J.M. Landsberg, \emph{On symmetric degeneracy loci, spaces of
  symmetric matrices of constant rank and dual varieties}, Math. Ann.
  \textbf{314} (1999), no.~1, 159--174.

\bibitem[LP80]{le_potier_stabili}
J.~Le~Potier, \emph{Stabilit\'e et amplitude sur {${\Bbb P}_{2}({\Bbb C})$}},
  Vector bundles and differential equations ({P}roc. {C}onf., {N}ice, 1979),
  Progr. Math., vol.~7, Birkh\"auser Boston, Mass., 1980, pp.~145--182.

\bibitem[MM05]{Manivel_Mezzetti}
L.~Manivel and E.~Mezzetti, \emph{On linear spaces of skew-symmetric matrices
  of constant rank}, Manuscripta Math. \textbf{117} (2005), no.~3, 319--331.

\bibitem[SU09]{sierra_ugaglia_gg}
J.C. Sierra and L.~Ugaglia, \emph{On globally generated vector bundles on
  projective spaces}, J. Pure Appl. Algebra \textbf{213} (2009), no.~11,
  2141--2146.

\bibitem[Syl86]{sylvester}
J.~Sylvester, \emph{On the dimension of spaces of linear transformations
  satisfying rank conditions}, Linear Algebra Appl. \textbf{78} (1986), 1--10.

\bibitem[Wes87]{Westwick1}
R.~Westwick, \emph{Spaces of matrices of fixed rank}, Linear and Multilinear
  Algebra \textbf{20} (1987), no.~2, 171--174.

\bibitem[Wes96]{Westwick}
\bysame, \emph{Spaces of matrices of fixed rank. {I}{I}}, Linear Algebra Appl.
  \textbf{235} (1996), 163--169.

\end{thebibliography}

\end{document}